\documentclass{amsart}

\usepackage[T1]{fontenc}
\usepackage{amssymb}
\usepackage{enumitem}
\usepackage{mathrsfs}
\usepackage[colorlinks, linkcolor=blue, citecolor=blue, urlcolor=blue]{hyperref}

\newtheorem{theorem}{Theorem}
\newtheorem{lemma}[theorem]{Lemma}

\theoremstyle{definition}
\newtheorem{definition}[theorem]{Definition}
\newtheorem{question}[theorem]{Question}

\DeclareMathOperator{\pr}{pr}
\DeclareMathOperator{\Cl}{Cl}

\title{On ordering of surjective cardinals}

\author{Guozhen Shen}
\address{Department of Philosophy (Zhuhai)\\
Sun Yat-sen University\\
No.~2 Daxue Road\\
Zhuhai 519082\\
Guangdong Province\\
People's Republic of China}
\email{shen\_guozhen@outlook.com}

\author{Wenjie Zhou}
\address{School of Philosophy\\
Wuhan University\\
No.~299 Bayi Road\\
Wuhan 430072\\
Hubei Province\\
People's Republic of China}
\email{3297835798@qq.com}



\begin{document}

\begin{abstract}
Let $\mathrm{Card}$ denote the class of cardinals. For all cardinals $\mathfrak{a}$ and $\mathfrak{b}$, $\mathfrak{a}\leqslant\mathfrak{b}$ means that there is an injection from a set of cardinality $\mathfrak{a}$ into a set of cardinality $\mathfrak{b}$, and $\mathfrak{a}\leqslant^\ast\mathfrak{b}$ means that there is a partial surjection from a set of cardinality $\mathfrak{b}$ onto a set of cardinality $\mathfrak{a}$. A doubly ordered set is a triple $\langle P,\preccurlyeq,\preccurlyeq^\ast\rangle$ such that $\preccurlyeq$ is a partial ordering on $P$, $\preccurlyeq^\ast$ is a preordering on $P$, and ${\preccurlyeq}\subseteq{\preccurlyeq^\ast}$. In 1966, Jech proved that for every partially ordered set $\langle P,\preccurlyeq\rangle$, there exists a model of $\mathsf{ZF}$ in which $\langle P,\preccurlyeq\rangle$ can be embedded into $\langle\mathrm{Card},\leqslant\rangle$. We generalize this result by showing that for every doubly ordered set $\langle P,\preccurlyeq,\preccurlyeq^\ast\rangle$, there exists a model of $\mathsf{ZF}$ in which $\langle P,\preccurlyeq,\preccurlyeq^\ast\rangle$ can be embedded into $\langle\mathrm{Card},\leqslant,\leqslant^\ast\rangle$.
\end{abstract}

\maketitle

\section{Introduction}

\begin{definition}
Let $\mathrm{Card}$ denote the class of cardinals. For all $\mathfrak{a},\mathfrak{b}\in\mathrm{Card}$,
\begin{enumerate}
  \item $\mathfrak{a}\leqslant\mathfrak{b}$ means that there is an injection
        from a set of cardinality $\mathfrak{a}$ into a set of cardinality $\mathfrak{b}$;
  \item $\mathfrak{a}\leqslant^\ast\mathfrak{b}$ means that there is a partial surjection
        from a set of cardinality $\mathfrak{b}$ onto a set of cardinality $\mathfrak{a}$.
\end{enumerate}
\end{definition}

Assuming the axiom of choice, it is well known that $\langle\mathrm{Card},\leqslant\rangle$ is a well-ordered class;
however, in its absence, we know only that it is a partially ordered class. In fact, a well-known theorem of Hartogs states
that the axiom of choice is equivalent to the assertion that $\langle\mathrm{Card},\leqslant\rangle$ is a linearly ordered class.
In 1966, Jech~\cite{Jech1966} proved that the structure $\langle\mathrm{Card},\leqslant\rangle$ can exhibit arbitrarily high complexity;
more precisely,

\begin{theorem}\label{SZ01}
Let $\mathfrak{M}$ be a model of $\mathsf{ZFC}$ and $\langle P,\preccurlyeq\rangle$ a partially ordered set in $\mathfrak{M}$.
Then there exists a model $\mathfrak{N}\supseteq\mathfrak{M}$ of $\mathsf{ZF}$ in which
$\langle P,\preccurlyeq\rangle$ can be embedded into $\langle\mathrm{Card},\leqslant\rangle$.
\end{theorem}

A natural question is whether the structure $\langle\mathrm{Card},\leqslant^\ast\rangle$ exhibits the same property.
In~\cite{Karagila2014}, Karagila proved the following generalization of Theorem~\ref{SZ01}.

\begin{theorem}\label{SZ02}
Let $\mathfrak{M}$ be a model of $\mathsf{ZFC}$, $\kappa$ an infinite cardinal in $\mathfrak{M}$,
and $\langle P,\preccurlyeq\rangle$ a partially ordered set in $\mathfrak{M}$.
Then there exists a model $\mathfrak{N}\supseteq\mathfrak{M}$ of $\mathsf{ZF}+\mathsf{DC}_\kappa$
in which $\langle P,\preccurlyeq\rangle$ can be embedded into both
$\langle\mathrm{Card},\leqslant\rangle$ and $\langle\mathrm{Card},\leqslant^\ast\rangle$.
\end{theorem}

Theorem~\ref{SZ02} falls short of optimality in the following sense.
Without assuming the axiom of choice, the best we can conclude is that
$\langle\mathrm{Card},\leqslant^\ast\rangle$ is a preordered class.
In fact, the assertion that $\langle\mathrm{Card},\leqslant^\ast\rangle$ is a partially ordered class
is the so called dual Cantor--Bernstein theorem (see \cite{BM1990}). This statement is not provable in $\mathsf{ZF}$,
and whether it implies the axiom of choice remains a long-standing open question.

\begin{definition}
A \emph{doubly ordered set} is a triple $\langle P,\preccurlyeq,\preccurlyeq^\ast\rangle$ such that
$\preccurlyeq$ is a partial ordering on $P$, $\preccurlyeq^\ast$ is a preordering on $P$, and
${\preccurlyeq}\subseteq{\preccurlyeq^\ast}$ (i.e., $p\preccurlyeq q$ implies $p\preccurlyeq^\ast q$ for all $p,q\in P$).
\end{definition}

In this article, we prove the following generalization of Theorem~\ref{SZ01},
and also hint at a proof of a generalization of Theorem~\ref{SZ02}.

\begin{theorem}\label{SZ03}
Let $\mathfrak{M}$ be a model of $\mathsf{ZFC}$ and $\langle P,\preccurlyeq,\preccurlyeq^\ast\rangle$ a doubly ordered set in $\mathfrak{M}$.
Then there exists a model $\mathfrak{N}\supseteq\mathfrak{M}$ of $\mathsf{ZF}$ in which $\langle P,\preccurlyeq,\preccurlyeq^\ast\rangle$
can be embedded into $\langle\mathrm{Card},\leqslant,\leqslant^\ast\rangle$.
\end{theorem}

\section{The main construction}
We shall employ the method of permutation models.
We refer the reader to~\cite[Chap.~8]{Halbeisen2025} or~\cite[Chap.~4]{Jech1973}
for an introduction to the theory of permutation models.
Permutation models are not models of $\mathsf{ZF}$;
they are models of $\mathsf{ZFA}$ (the Zermelo--Fraenkel set theory with atoms).
Given a doubly ordered set $\langle P,\preccurlyeq,\preccurlyeq^\ast\rangle$ in the ground model,
we shall construct a permutation model in which there exists a family of sets $\langle S_p\rangle_{p\in P}$
such that, for all $p,q\in P$, $p\preccurlyeq q$ if and only if $|S_p|\leqslant|S_q|$,
and $p\preccurlyeq^\ast q$ if and only if $|S_p|\leqslant^\ast|S_q|$.
Then, by the Jech--Sochor transfer theorem (see~\cite[Theorem~17.1]{Halbeisen2025} or~\cite[Theorem~6.1]{Jech1973}),
we arrive at Theorem~\ref{SZ03}.

For any quadruple $\langle x_0,x_1,x_2,x_3\rangle$ and any $i<4$, define $\pr_i(\langle x_0,x_1,x_2,x_3\rangle)=x_i$.
For any set $S$, let $\mathscr{S}(S)$ denote the set of all permutations of $S$.

Let $\langle P,\preccurlyeq,\preccurlyeq^\ast\rangle$ be a doubly ordered set in the ground model.
We construct the set $A$ of atoms and a group $\mathcal{G}$ of permutations of $A$ as follows. Define by recursion
\begin{align*}
A_0     & =\{\langle0,p,\varnothing,k\rangle\mid p\in P\text{ and }k\in\omega\};\\
A_{n+1} & =A_n\cup\{\langle n+1,q,a,0\rangle\mid q\in P,a\in A_n\text{ and }\pr_1(a)\prec q\}\\
        & \phantom{{}=A_n}\cup
          \{\langle n+1,q,a,k\rangle\mid q\in P,a\in A_n,\pr_1(a)\not\preccurlyeq q,\pr_1(a)\preccurlyeq^\ast q\text{ and }k\in\omega\}.
\end{align*}
Define by recursion a group $\mathcal{G}_n$ of permutations of $A_n$ as follows:
\begin{itemize}
  \item $\mathcal{G}_0=\{f\in\mathscr{S}(A_0)\mid\pr_1(f(a))=\pr_1(a)\text{ for all }a\in A_0\}$;
  \item For all $f\in\mathscr{S}(A_{n+1})$, $f\in\mathcal{G}_{n+1}$ if and only if $f{\upharpoonright}A_n\in\mathcal{G}_n$
        and, for all $b\in A_{n+1}\setminus A_n$, $\pr_1(f(b))=\pr_1(b)$ and $\pr_2(f(b))=f(\pr_2(b))$.
\end{itemize}
Let $A=\bigcup_{n\in\omega}A_n$ and let
$\mathcal{G}=\{\pi\in\mathscr{S}(A)\mid\pi{\upharpoonright}A_n\in\mathcal{G}_n\text{ for all }n\in\omega\}$.
Let $\mathcal{V}$ be the permutation model determined by $\mathcal{G}$ and finite supports;
that is, $x\in\mathcal{V}$ if and only if $x\subseteq\mathcal{V}$ and $x$ has a \emph{finite support}---a finite
subset $B\subseteq A$ such that every permutation $\pi\in\mathcal{G}$ fixing $B$ pointwise also fixes $x$.
Note that, for every $n\in\omega$, $A_n$ is fixed by every permutation in $\mathcal{G}$, so $A_n\in\mathcal{V}$.
For each $p\in P$, let $S_p=\{a\in A\mid\pr_1(a)=p\}$.

\begin{lemma}\label{sh04}
In $\mathcal{V}$, for all $p,q\in P$, if $p\preccurlyeq q$, then $|S_p|\leqslant|S_q|$,
and if $p\preccurlyeq^\ast q$, then $|S_p|\leqslant^\ast|S_q|$.
\end{lemma}
\begin{proof}
In $\mathcal{V}$, for all $p,q\in P$, if $p\prec q$, then the function that maps each $a\in S_p$
to $\langle n+1,q,a,0\rangle$, where $n$ is the least natural number such that $a\in A_n$,
is an injection from $S_p$ into $S_q$, and if $p\not\preccurlyeq q$ and $p\preccurlyeq^\ast q$, then
the function that maps each $b\in S_q$ to $\pr_2(b)$ is a surjection from $S_q$ onto a superset of $S_p$.
\end{proof}

We say that a subset $C$ of $A$ is \emph{closed} if for all $n\in\omega$, all $a\in C\cap A_n$, and all $q\in P$ with $\pr_1(a)\prec q$,
$\langle n+1,q,a,0\rangle\in C$, and for all $b\in C\setminus A_0$, $\pr_2(b)\in C$.
The \emph{closure} of $B\subseteq A$, denoted by $\Cl(B)$, is the least closed set that includes $B$.
It is easy to see that finite unions of closed subsets of $A$ are closed, and hence $\Cl(B\cup C)=\Cl(B)\cup\Cl(C)$ for all $B,C\subseteq A$.

\begin{lemma}\label{sh01}
For every finite subset $B\subseteq A$, $\{b\in\Cl(B)\mid\pr_3(b)\neq0\}$ is finite.
\end{lemma}
\begin{proof}
Let $B$ be a finite subset of $A$. Define by recursion
\begin{align*}
X^B_0     & =\varnothing,\\
Y^B_0     & =B;\\
X^B_{m+1} & =X^B_m\cup\{\langle n+1,q,a,0\rangle\mid q\in P,a\in(X^B_m\cup Y^B_m)\cap A_n\text{ and }\pr_1(a)\prec q\},\\
Y^B_{m+1} & =Y^B_m\cup\{\pr_2(b)\mid b\in Y^B_m\setminus A_0\}.
\end{align*}
It is easy to see that $\Cl(B)=\bigcup_{m\in\omega}(X^B_m\cup Y^B_m)$.
An easy induction shows that $\pr_3(b)=0$ for all $b\in\bigcup_{m\in\omega}X^B_m$.
Hence, it suffices to show that $\bigcup_{m\in\omega}Y^B_m$ is finite.

We prove by induction on $n\in\omega$ that $\bigcup_{m\in\omega}Y^B_m$ is finite for every finite subset $B\subseteq A_n$.
For $n=0$, an easy induction shows that $Y^B_m=B$ for all $m\in\omega$,
and thus $\bigcup_{m\in\omega}Y^B_m=B$ is finite. Let $B$ be a finite subset of $A_{n+1}$.
Let $B'=\{\pr_2(b)\mid b\in B\setminus A_0\}$. Then $B'$ is a finite subset of~$A_n$.
An easy induction shows that $\bigcup_{m\in\omega}Y^B_m\subseteq B\cup\bigcup_{m\in\omega}Y^{B'}_m$.
Hence, the conclusion follows from the induction hypothesis.
\end{proof}

\begin{lemma}\label{sh02}
For every closed subset $C\subseteq A$ and every $m\in\omega$,
every $g\in\mathcal{G}_m$ that fixes $C\cap A_m$ pointwise extends to a permutation $\pi\in\mathcal{G}$ that fixes $C$ pointwise.
\end{lemma}
\begin{proof}
We define by recursion $f_n\in\mathcal{G}_{m+n}$ that fixes $C\cap A_{m+n}$ pointwise as follows. Let $f_0=g$.
Let $n\in\omega$ and assume that $f_n\in\mathcal{G}_{m+n}$ is defined and fixes $C\cap A_{m+n}$ pointwise.
Let $f_{n+1}$ be the function on $A_{m+n+1}$ that extends $f_n$ and such that
\begin{itemize}
  \item for all $q\in P$ and all $a\in A_{m+n}$ such that $\pr_1(a)\prec q$,
        \[
        f_{n+1}(\langle m+n+1,q,a,0\rangle)=\langle m+n+1,q,f_n(a),0\rangle;
        \]
  \item for all $q\in P$, all $a\in A_{m+n}$ such that $\pr_1(a)\not\preccurlyeq q$
        and $\pr_1(a)\preccurlyeq^\ast q$, and all $k\in\omega$,
        \[
        f_{n+1}(\langle m+n+1,q,a,k\rangle)=\langle m+n+1,q,f_n(a),k\rangle.
        \]
\end{itemize}
It is easy to see that $f_{n+1}\in\mathcal{G}_{m+n+1}$ fixes $C\cap A_{m+n+1}$ pointwise.
Finally, it suffices to take $\pi=\bigcup_{n\in\omega}f_n$.
\end{proof}

\begin{lemma}\label{sh03}
For every finite subset $B\subseteq A$ and every $c\in A\setminus\Cl(B)$,
there exists a permutation $\pi\in\mathcal{G}$ that fixes $B$ pointwise and moves $c$.
\end{lemma}
\begin{proof}
Let $B$ be a finite subset of $A$.
We prove by induction on $n\in\omega$ that for all $c\in A_n\setminus\Cl(B)$,
there exists a permutation $\pi\in\mathcal{G}$ that fixes $B$ pointwise and moves $c$.

Let $c\in A_0\setminus\Cl(B)$. Then $c=\langle0,p,\varnothing,k\rangle$ for some $p\in P$ and $k\in\omega$.
By Lemma~\ref{sh01}, there exists an $l\in\omega\setminus\{k\}$ such that $l\neq\pr_3(b)$ for all $b\in\Cl(B)$.
Let $d=\langle0,p,\varnothing,l\rangle$ and let $g\in\mathscr{S}(A_0)$ be the transposition that swaps $c$ and $d$.
Then $g\in\mathcal{G}_0$ fixes $\Cl(B)\cap A_0$ pointwise. By Lemma~\ref{sh02},
$g$ extends to a permutation $\pi\in\mathcal{G}$ that fixes $\Cl(B)$ pointwise.
Then $\pi$ fixes $B$ pointwise and moves $c$ to $d$.

Let $n\in\omega$ and let $c\in A_{n+1}\setminus(A_n\cup\Cl(B))$. Consider the following two cases.

\textsc{Case}~1. $c=\langle n+1,q,a,0\rangle$ for some $q\in P$ and $a\in A_n$ with $\pr_1(a)\prec q$.
Since $c\notin\Cl(B)$, it~follows that $a\notin\Cl(B)$. By the induction hypothesis,
there exists a permutation $\pi\in\mathcal{G}$ that fixes $B$ pointwise and moves $a$.
Hence, $\pr_2(\pi(c))=\pi(\pr_2(c))=\pi(a)\neq a=\pr_2(c)$, which implies that $\pi$ moves $c$.

\textsc{Case}~2. $c=\langle n+1,q,a,k\rangle$ for some $q\in P$ and $a\in A_n$ such that
$\pr_1(a)\not\preccurlyeq q$ and $\pr_1(a)\preccurlyeq^\ast q$ and some $k\in\omega$.
By Lemma~\ref{sh01}, there exists an $l\in\omega\setminus\{k\}$ such that $l\neq\pr_3(b)$ for all $b\in\Cl(B)$.
Let $d=\langle n+1,q,a,l\rangle$ and let $g\in\mathscr{S}(A_{n+1})$ be the transposition that swaps $c$ and~$d$.
Then $g\in\mathcal{G}_{n+1}$ fixes $\Cl(B)\cap A_{n+1}$ pointwise. By Lemma~\ref{sh02},
$g$ extends to a permutation $\pi\in\mathcal{G}$ that fixes $\Cl(B)$ pointwise.
Then $\pi$ fixes $B$ pointwise and moves $c$ to $d$.
\end{proof}

\begin{lemma}\label{sh05}
In $\mathcal{V}$, for all $p,q\in P$, if $p\not\preccurlyeq^\ast q$, then $|S_p|\nleqslant^\ast|S_q|$.
\end{lemma}
\begin{proof}
Assume towards a contradiction that there exists a surjection $h\in\mathcal{V}$ from $S_q$ onto $S_p$.
Let $B\subseteq A$ be a finite support of $h$. By Lemma~\ref{sh01}, there is a $k\in\omega$ such that
$c=\langle0,p,\varnothing,k\rangle\notin\Cl(B)$. Since $c\in S_p$ and $h$ is surjective,
there exists a $b\in S_q$ such that $h(b)=c$. It is easy to see that for all $a\in\Cl(\{b\})$,
either $\pr_0(a)\neq0$ or $\pr_1(a)\preccurlyeq^\ast q$. Hence, $c\notin\Cl(\{b\})$,
which implies that $c\notin\Cl(B)\cup\Cl(\{b\})=\Cl(B\cup\{b\})$.
By Lemma~\ref{sh03}, there exists a permutation $\pi\in\mathcal{G}$ that fixes $B\cup\{b\}$ pointwise and moves $c$.
Then $\pi$ fixes $h$ and $b$ but moves $c$, which is a contradiction.
\end{proof}

\begin{lemma}\label{sh06}
In $\mathcal{V}$, for all $p,q\in P$, if $p\not\preccurlyeq q$, then $|S_p|\nleqslant|S_q|$.
\end{lemma}
\begin{proof}
Assume towards a contradiction that there exists an injection $h\in\mathcal{V}$ from $S_p$ into $S_q$.
Let $B\subseteq A$ be a finite support of $h$. By Lemma~\ref{sh01}, there is a $k\in\omega$ such that
$c=\langle0,p,\varnothing,k\rangle\notin\Cl(B)$. Let $b=h(c)$. If $c\notin\Cl(\{b\})$,
then $c\notin\Cl(B)\cup\Cl(\{b\})=\Cl(B\cup\{b\})$, and hence it follows from Lemma~\ref{sh03}
that there exists a permutation $\pi\in\mathcal{G}$ that fixes $B\cup\{b\}$ pointwise and moves $c$.
Then $\pi$ fixes $h$ and $b$ but moves $c$, contradicting the injectivity of $h$.
Otherwise, $c\in\Cl(\{b\})$, and hence $b\notin\Cl(B)$ since $c\notin\Cl(B)$.
It is easy to see that for all $a\in\Cl(\{c\})$, $p\preccurlyeq\pr_1(a)$. Hence, $b\notin\Cl(\{c\})$,
which implies that $b\notin\Cl(B)\cup\Cl(\{c\})=\Cl(B\cup\{c\})$.
By Lemma~\ref{sh03}, there exists a permutation $\pi\in\mathcal{G}$ that fixes $B\cup\{c\}$ pointwise and moves $b$.
Then $\pi$ fixes $h$ and $c$ but moves $b$, which is also a contradiction.
\end{proof}

\begin{theorem}\label{sh07}
In $\mathcal{V}$, for all $p,q\in P$, $p\preccurlyeq q$ if and only if $|S_p|\leqslant|S_q|$,
and $p\preccurlyeq^\ast q$ if and only if $|S_p|\leqslant^\ast|S_q|$.
\end{theorem}
\begin{proof}
Immediately follows from Lemmas~\ref{sh04}, \ref{sh05} and~\ref{sh06}.
\end{proof}

Finally, Theorem~\ref{SZ03} follows from Theorem~\ref{sh07} and the Jech--Sochor transfer theorem.

\section{Concluding remarks}
We conclude this article with a further generalization of Theorem~\ref{SZ03} and several directions for future research.

By replacing $\omega$ with $\kappa^+$ in the proof of Theorem~\ref{SZ03},
using supports of cardinality at most $\kappa$ instead of finite supports,
and applying a transfer theorem of Pincus \cite[Theorem~4]{Pincus1977},
we can readily establish the following theorem, which also generalizes Theorem~\ref{SZ02}.

\begin{theorem}
Let $\mathfrak{M}$ be a model of $\mathsf{ZFC}$, $\kappa$ an infinite cardinal in $\mathfrak{M}$,
and $\langle P,\preccurlyeq,\preccurlyeq^\ast\rangle$ a doubly ordered set in $\mathfrak{M}$.
Then there exists a model $\mathfrak{N}\supseteq\mathfrak{M}$ of $\mathsf{ZF}+\mathsf{DC}_\kappa$
in which $\langle P,\preccurlyeq,\preccurlyeq^\ast\rangle$ can be embedded into $\langle\mathrm{Card},\leqslant,\leqslant^\ast\rangle$.
\end{theorem}

We say that a cardinal $\mathfrak{a}$ is \emph{Dedekind finite} if $\aleph_0\nleqslant\mathfrak{a}$,
and $\mathfrak{a}$ is \emph{power Dedekind finite} if $\aleph_0\nleqslant2^\mathfrak{a}$,
or equivalently, by Kuratowski's theorem (see~\cite[Proposition~5.3]{Halbeisen2025}), if $\aleph_0\nleqslant^\ast\mathfrak{a}$.
Let $\mathrm{DFCard}$ and $\mathrm{PDFCard}$ denote the classes of all Dedekind finite cardinals
and all power Dedekind finite cardinals, respectively.
Note that in the proof of Theorem~\ref{SZ03}, the cardinals $|S_p|$ (for $p\in P$) may be Dedekind infinite.

\begin{question}
Let $\mathfrak{M}$ be a model of $\mathsf{ZFC}$ and $\langle P,\preccurlyeq,\preccurlyeq^\ast\rangle$ a doubly ordered set in $\mathfrak{M}$.
Is there a model $\mathfrak{N}\supseteq\mathfrak{M}$ of $\mathsf{ZF}$ in which $\langle P,\preccurlyeq,\preccurlyeq^\ast\rangle$ can be
embedded into $\langle\mathrm{DFCard},\leqslant,\leqslant^\ast\rangle$ or even $\langle\mathrm{PDFCard},\leqslant,\leqslant^\ast\rangle$?
\end{question}

The notion of a cardinal algebra was initiated by Tarski in his influential 1949 book~\cite{Tarski1949},
and the notion of a weak cardinal algebra was introduced simultaneously and independently by Bhaskara Rao and Shortt
on the one hand, and by Wehrung on the other hand in \cite{Rao1992,Wehrung1992}.
It is clear that, assuming $\mathsf{AC}_\omega$, the cardinals form a cardinal algebra.
It is proved in \cite[Theorem~2.11]{JS2025} that, assuming $\mathsf{AC}_\omega$,
the surjective cardinals form a weak cardinal algebra (referred to there as the \emph{surjective cardinal algebra}).
On the other hand, \cite[Theorem~3.5]{JS2025} shows that $\mathsf{ZF}+\mathsf{DC}_\kappa$
cannot prove that the surjective cardinals form a cardinal algebra.

\begin{question}
Let $\mathfrak{M}$ be a model of $\mathsf{ZFC}$ and $A$ a cardinal algebra in $\mathfrak{M}$.
Is there a model $\mathfrak{N}\supseteq\mathfrak{M}$ of $\mathsf{ZF}+\mathsf{AC}_\omega$ in which
$A$ can be embedded into $\mathrm{Card}$?
\end{question}

\begin{question}
Let $\mathfrak{M}$ be a model of $\mathsf{ZFC}$ and $A$ a weak cardinal algebra in $\mathfrak{M}$.
Is there a model $\mathfrak{N}\supseteq\mathfrak{M}$ of $\mathsf{ZF}+\mathsf{AC}_\omega$ in which
$A$ can be embedded into the class of surjective cardinals?
\end{question}


\begin{thebibliography}{99}
    \bibitem{BM1990} B.~Banaschewski and G.\,H.~Moore.
		{\itshape The dual Cantor--Bernstein theorem and the partition principle.}
		Notre Dame J.\ Form.\ Log.\
		31 (1990) pp.\ 375--381.

	\bibitem{Halbeisen2025} L.~Halbeisen.
		{\itshape Combinatorial Set Theory: With a Gentle Introduction to Forcing.} 3rd edition.
		In: {\itshape Springer Monographs in Mathematics}. Springer, Cham, 2025.

	\bibitem{Jech1966} T.~Jech.
		{\itshape On ordering of cardinalities.}
		Bull.\ Pol.\ Acad.\ Sci.\ Math.\
		14 (1966) pp.\ 293--296.

	\bibitem{Jech1973} T.~Jech.
		{\itshape The Axiom of Choice.}
		In: {\itshape Studies in Logic and the Foundations of Mathematics}, Vol. 75. North-Holland, Amsterdam, 1973.

	\bibitem{JS2025} J.~Jin and G.~Shen.
		{\itshape A note on surjective cardinals.}
		accepted in J.\ Symb.\ Log.\ (2025).

	\bibitem{Karagila2014} A.~Karagila.
		{\itshape Embedding orders into the cardinals with $\mathsf{DC}_\kappa$.}
		Fund.\ Math.\
		226 (2014) pp.\ 143--156.
	
	\bibitem{Pincus1977} D.~Pincus.
		{\itshape Adding dependent choice.}
		Ann.\ Math.\ Logic
		11 (1977) pp.\ 105--145.

	\bibitem{Rao1992} K.\,P.\,S.~Bhaskara Rao and R.\,M.~Shortt.
		{\itshape Weak cardinal algebras.}
		Ann.\ New York Acad.\ Sci.\
		659 (1992) pp.\ 156--162.

	\bibitem{Tarski1949} A.~Tarski.
		{\itshape Cardinal Algebras.}
		Oxford University Press, New York, 1949.

	\bibitem{Wehrung1992} F.~Wehrung.
		{\itshape Injective positively ordered monoids I.}
		J.\ Pure Appl.\ Algebra
		83 (1992) pp.\ 43--82.
\end{thebibliography}
\end{document}